\documentclass[10pt,a4paper]{article}
\usepackage{url}
\usepackage{amssymb}

\def\pg{\mbox{\rm PG}}
\def\ag{\mbox{\rm AG}}
\def\gf{\mbox{\rm GF}}

\def\C{\mathcal{C}}
\def\S{\mathcal{S}}
\def\P{\mathcal{P}}
\def\B{\mathcal{B}}
\def\I{\mathrel{\mathrm I}}
\def\O{\mathcal{O}}

\def\K{\mathcal{K}}

\def\IS{\S=(\P,\B,\I)}

\newtheorem{theo}{Theorem}
\newtheorem{lemma}[theo]{Lemma}
\newtheorem{cor}[theo]{Corollary}
\newtheorem{definition}[theo]{Definition}

\newtheorem{conjecture}[theo]{Conjecture}

\newenvironment{proof}{\noindent{\bf Proof. }}{\ignorespaces\rule{1pt}{0pt}\hfill $\square$\medskip}

\begin{document}

\title{On the structure of the directions not determined by a large affine point set\footnote{Final version, to appear in J. Algebraic Combin.}}
\author{Jan De Beule\thanks{This author is a Postdoctoral Fellow of the Research Foundation
Flanders -- Belgium (FWO-Vlaanderen). This author also acknowledges the Research Foundation Flanders -- Belgium (FWO-Vlaanderen)
for a travel grant.},
\addtocounter{footnote}{2}%
P\'eter Sziklai\thanks{Partially supported by OTKA T-49662,67867, K-81310, T\'AMOP, ERC, and Bolyai grants.},%
\addtocounter{footnote}{-3}
 and Marcella Tak\'ats\thanks{This research was partially supported by OTKA K-81310 grant.}}
\date{}

\maketitle
\begin{abstract}
Given a point set $U$ in an $n$-dimensional affine space of size $q^{n-1}-\varepsilon$, we obtain
information on the structure of the set of directions that are not determined by $U$, and we describe an
application in the theory of partial ovoids of certain partial geometries.
\end{abstract}

{\bf keywords:} directions in an affine space, partial ovoids, partial geometries
{\bf MSC (2010)}: 05B25, 51D20, 51E14, 51E20, 51E21.

\section{Introduction}

Let $\pg(n,q)$ and $\ag(n,q)$ denote the projective and the affine
$n$-dimensional space over the finite field $\gf(q)$ of $q$ elements.
Given a point set $U\subset \ag(n,q)\subset \pg(n,q)$, a direction, i.e. a point
$t \in H_\infty = \pg(n,q)\setminus \ag(n,q)$ is {\it determined} by $U$ if
there is an affine line through $t$ which contains at least 2 points of $U$.
Note that if $|U|>q^{n-1}$ then every direction is determined.

Especially in the planar case, many results on extendability of affine point sets 
not determining a given set of directions are known. We mention the following 
theorem from \cite{Szonyi1996}.

\begin{theo}
Let $U \subseteq \ag(2,q)$ be a set of affine points of size $q-\varepsilon > q - \sqrt{q}/2$, which
does not determine a set $D$ of more than $(q+1)/2$ directions. Then $U$ can be extended
to a set of size $q$, not determining the set $D$ of directions.
\end{theo}

An extendability result known for general dimension is the following. Originally, it
was proved in \cite{DBG2008} for $n=3$. A proof for general $n$ can be found
in \cite{Ball:2011}.

\begin{theo}\label{th:DBG_Ball}
Let $q=p^h$, $p$ an odd prime and $h > 1$, and let $U \subseteq \ag(n,q)$, $n \geq 3$,
be a set of affine points of size $q^{n-1}-2$, which does not determine a set $D$ of at least
$p+2$ directions. Then $U$ can be extended to a set of size $q$, not determining the
set $D$ of directions.
\end{theo}

The natural question is whether Theorem~\ref{th:DBG_Ball} can be improved in the sense
that extendability of sets of size $q^{n-1} -\varepsilon$ is investigated, for $\varepsilon > 2$,
possibly with stronger assumptions on the number of non-determined directions. This general
question seems to be hard for $n \geq 3$, and up to our knowledge, no other result 
different from Theorem~\ref{th:DBG_Ball} is known for $n \geq 3$.

In this paper, we investigate affine point sets of size $q^{n-1} -\varepsilon$, for arbitrary
$\varepsilon$, where the strongest results are obtained when $\varepsilon$ is small.
Instead of formulating an extendability result in terms of the number of non-determined directions,
we formulate it in terms of the structure of the set of non-determined directions.
Finally, we add a section with an application of the obtained theorem.

\section{The main result}

A point of $\pg(n,q)$ is represented by a homogenous $(n+1)$-tuple $(a_0,
a_1,...,a_n) \neq (0,0,\dots,0)$. A hyperplane is the set of points whose
coordinates satisfy a linear equation
\[
a_0X_0 + a_1X_1 + \dots + a_nX_n = 0 \,
\]
and so hyperplanes are represented by a homogenous $(n+1)$-tuple $[a_0,
a_1,...,a_n] \neq [0,0,\dots,0]$. Embed the affine space $\ag(n,q)$ in $\pg(n,q)$
such that the hyperplane $X_0 = 0$, i.e. the hyperplane with coordinates
$[1,0,\dots,0]$ is the hyperplane at infinity of $\ag(n,q)$. Then the points of $\ag(n,q)$
will be coordinatized as $(1, a_1, a_2,...,a_n)$.

The map $\delta$ from the points of $\pg(n,q)$ to its hyperplanes, mapping a point
$(a_0, a_1, a_2,...,a_n)$ to a hyperplane $[a_0,a_1,\dots,a_n]$ is the standard duality of $\pg(n,q)$.

Let $U \subseteq \ag(n,q)$ be an affine point set, $|U|=q^{n-1} - \varepsilon$.  Embed $\ag(n,q)$ in
$\pg(n,q)$ and denote the hyperplane at infinity as $H_{\infty}$. Let $D\subseteq H_\infty$ be the set
of directions determined by $U$ and put $N=H_\infty\setminus D$ the set of non-determined
directions.

\begin{lemma}\label{contnum}
Let $0\leq r\leq n-2$. Let $\alpha=(0,\alpha_1,\alpha_2,\alpha_3,...,\alpha_n) \in N$
be a non-determined direction. Then each of the affine subspaces of dimension $r+1$
through $\alpha$ contain at most $q^{r}$ points of $U$.
\end{lemma}

\begin{proof}
We prove it by the pigeon hole principle. An affine subspace of dimension $r+1$ through $\alpha$
contains $q^{r}$ affine (disjoint) lines through $\alpha$, and each line contains at most one point of
$U$ as $\alpha$ is a non-determined direction.
\end{proof}

\begin{definition}
If an affine subspace of dimension $r+1\leq n-1$ through $\alpha \in N$ contains less
than $q^r$ points of $U$, then it is called a {\em deficient subspace}. If
it contains $q^r-t$ points of $U$, then its {\em deficiency} is $t$.
\end{definition}

\begin{cor}
\label{epsdefsub}
Let $T \subseteq H_\infty$ be a subspace of dimension $r \leq n-2$ containing $\alpha \in N$.
Then there are precisely $\varepsilon$ deficient subspaces of dimension $r+1$
(counted possibly with multiplicity) through $T$ (a subspace with deficiency $t$ is counted
with multiplicity $t$).
\end{cor}

In particular:

\begin{cor}
\label{epsline}
There are precisely $\varepsilon$ affine {\em lines} through $\alpha$ not containing
any point of $U$ (and $q^{n-1}-\varepsilon$ lines with $1$ point of $U$ each).
\end{cor}

Now consider the set $U=\{(1,a^i_{1},a^i_{2},a^i_{3}, \dots ,a^i_{n}) : i=1,...,q^{n-1}-\varepsilon\}$.
We define its R\'edei polynomial as follows:
\[
R(X_0,X_1,X_2,...,X_n)=\prod_{i=1}^{q^{n-1}-\varepsilon} (X_0+a^i_{1}X_1+a^i_{2}X_2+...+a^i_{n}
X_n).
\]

The intersection properties of of the set $U$ with hyperplanes of $\pg(n,q)$ are translated into
algebraic properties of the polynomial $R$ as follows. Consider $x_1,x_2,\dots,x_n \in \gf(q)$,
then $x \in \gf(q)$ is a root with multiplicity $m$ of the equation $R(X_0,x_1,x_2,\dots,x_n) = 0$
if and only if the hyperplane $[x,x_1,x_2,\dots,x_n]$ contains $m$ points of $U$.

Define the set $S(X_1,X_2,...,X_n)=\{ a^i_{1}X_1+a^i_{2}X_2+...+a^i_{n}X_n :
i=1, ..., q^{n-1}-\varepsilon\}$, then $R$ can be written as
$$R(X_0,X_1,X_2,...,X_n) = \sum_{j=0}^{q^{n-1}-\varepsilon} \sigma_{q^{n-1}-\varepsilon-j}
(X_1,X_2,...,X_n) X_0^j,$$
where $\sigma_j(X_1,X_2,...,X_n)$ is the $j$-th elementary symmetric polynomial of the set $S
(X_1,X_2,...,X_n)$.

Consider the subspace $s_{x_1,x_2,...,x_n}\subset H_\infty=[1,0,...,0]$ of dimension $n-2$ which is
the intersection of the hyperplanes $[x_0,x_1,x_2,...,x_n], x_0\in\gf(q)$. Suppose that $s_
{x_1,x_2,...,x_n}$ contains an undetermined direction then, by Lemma \ref{contnum}, each of the
hyperplanes different from $H_{\infty}$ through $s_{x_1,x_2,...,x_n}$, contains at most $q^{n-2}$
points of $U$. 
This implies that there are precisely $\varepsilon$ such hyperplanes (counted
with multiplicity) through $s_{x_1,x_2,...,x_n}$ containing less than $q^{n-2}$ points of $U$
(a hyperplane with deficiency $t$ is counted with multiplicity $t$).  Algebraically, this means that
for the $(n-2)$-dimensional subspace $s_{x_1,x_2,...,x_n}$,
\begin{equation}\label{eq:star}
R(X_0,x_1,x_2,...,x_n)f(X_0)= (X_0^q-X_0)^{q^{n-2}}
\end{equation}
where $f(X_0) = X_0^{\varepsilon}+ \sum_{k=1}^{\varepsilon} f_kX_0^{\varepsilon - k}$ is a fully
reducible polynomial of degree $\varepsilon$. Comparing the two sides
of equation~(\ref{eq:star}), one gets linear equations for the coefficients $f_k$ of $f$ in terms of
the $\sigma_j(x_1,\dots,x_n)$, and it is easy to see that the solutions for each $f_k$ are a polynomial
expression in terms of the $\sigma_j(x_1,\dots,x_n)$, $j=1,\dots, k$, use e.g. Cramer's rule to solve
the system of equations, and notice that the determinant in the denominator equals $1$.
The polynomial expression is
independent from the elements $x_1,x_2,\dots,x_n$ (still under the assumption that $s_{x_1,x_2,\dots,x_n}$
does contain an undetermined direction), so we can change them for the variables $X_1,X_2,\dots,X_n$ which makes the
coefficients $f_k$ polynomials in these variables; then
the total degree of each $f_k(\sigma_j(X_1,\dots,X_n)\ :\ j=1,\dots,n)$ is $k$.

Hence, using the polynomial expressions $f_k(\sigma_j : j)$, we can define the polynomial
\begin{equation}\label{eq:f}
f(X_0,X_1,\dots,X_n) = X_0^{\varepsilon}+\sum_{k=1}^{\varepsilon}f_k({\sigma_1,\dots,\sigma_k})X_0^{\varepsilon - k}
\end{equation}

Clearly, $f(X_0,X_1,\dots,X_n)$ is a polynomial of total degree $\varepsilon$, and, substituting
$X_i=x_i$, $i=1,\dots,n$ for which $s_{x_1,\dots,x_n}$ contains an undetermined direction, yields
the polynomial $f(X_0,x_1,\dots,x_n)$ that splits completely into $\varepsilon$ linear factors. Also, since
$f$ contains the term $X_0^{\varepsilon}$, the point $(1,0,0,\dots,0)$ is not a point of the hypersurface.

Suppose now that $f = \prod_{i} \phi_i$, where the polynomials $\phi_i(X_1,\dots,X_n)$ are irreducible of degree
$\varepsilon_i$, $\sum_{i}\varepsilon_i = \varepsilon$. Then each factor inherits the properties
that (i) whenever the subspace $s_{x_1,x_2,...,x_n}\subset H_\infty$ of dimension $n-2$
contains an undetermined direction, then $\phi_i(X_0,x_1,x_2,...,x_n)$ splits into $\varepsilon_i$ linear
factors; and (ii) $(1,0,...,0)$ is not a point of $\phi_i$. So from now on we will think of $f$ as an
irreducible polynomial satisfying (i) and (ii).

$f(X_0,X_1,\dots,X_n)=0$ is an algebraic hypersurface in the dual space $\pg(n,q)$.
Our aim is to prove that it  splits into $\varepsilon$ hyperplanes, or (equivalently) that it contains
a linear factor (i.e. a hyperplane; then we can decrease $\varepsilon$ by one, etc.). Therefore,
we state and prove a series of technical lemmas.

\begin{lemma}\label{defhyperpl}
Let $T \neq H_{\infty}$ be a deficient hyperplane through $\alpha = (\alpha_0,\alpha_1,\dots,\alpha_n)
\in N$ (so $T$ contains less than $q^{n-2}$ points of $U$). Then in the dual space $\pg(n,q)$, $T$
corresponds to an intersection point $t$ of $f$ and the hyperplane $[\alpha_0,\alpha_1,\dots,\alpha_n]$.
\end{lemma}
\begin{proof}
If $T = [x_0,x_2,\dots,x_n]$ is a deficient hyperplane, then $x_0$ is a solution
of the equation $f(X_0,x_1,x_2,\dots,x_n) = 0$, hence, in the dual space $\pg(n,q)$, $t=(x_0,x_1,\dots,x_n)$
is a point of $f$. If  $T$ contains $\alpha = (\alpha_0,\alpha_1,\dots,\alpha_n) \in N$, then $t$ is contained
in the hyperplane $[\alpha_0,\alpha_1,\alpha_2,\dots,\alpha_n]$.
\end{proof}

\begin{lemma}\label{intersection}
Let $(\alpha)\in N$ be a non-determined direction. Then in the dual space $\pg(n,q)$ the
intersection of the hyperplane $[\alpha]$ and $f$ is precisely the union of $\varepsilon$
different subspaces of dimension $n-2$.
\end{lemma}
\begin{proof}
First notice that
\begin{quote}
If $(0,\alpha_1,\alpha_2,...,\alpha_n)\in H_\infty=[1,0,...,0]$ is an undetermined direction, then for
all the subspaces $s_{x_1,x_2,...,x_n}\subset H_\infty$ of dimension $n-2$
through $(0,\alpha_1,\alpha_2,\alpha_3,...,\alpha_n)$ the polynomial $f(X_0,x_1,x_2,...,x_n)$ has
precisely $\varepsilon$ roots, counted with multiplicity.
\end{quote}
translates to
\begin{quote}
In the hyperplane $[0,\alpha_1,\alpha_2,...,\alpha_n]\ni (1,0,...,0)$, all the lines through
$(1,0,...,0)$ intersect the surface $f(X_0,x_1,x_2,...,x_n)=0$
in precisely $\varepsilon$ points, counted with intersection multiplicity.
\end{quote}
Define $\bar f$ as the surface of degree $\bar{\varepsilon}\leq\varepsilon$, which is the
intersection of $f$ and the hyperplane $[0,\alpha_1,\alpha_2,...,\alpha_n]$. We know that all the
lines through $(1,0,...,0)$ intersect $\bar f$
in precisely $\varepsilon$ points (counted with intersection multiplicity). So if $\bar{f}=\prod_i \bar{\phi}
_i$, where $\bar{\phi}_i$ is irreducible of degree $\bar{\varepsilon}_i$ and $\sum_i \bar{\varepsilon}_i
=\bar{\varepsilon}$, then we have that all the lines through $(1,0,...,0)$ intersect $\bar{\phi}_i$
in precisely $\bar{\varepsilon}_i$ points (counted with intersection multiplicity).

By {Corollary \ref{epsline}} we know that there are precisely $\varepsilon$ different affine lines
through the non-determined direction $(\alpha)$ not containing any point of $U$. In the dual space
$\pg(n,q)$ these lines correspond to $\varepsilon$ different subspaces of dimension $n-2$
contained in the hyperplane $[\alpha]$. The deficient hyperplanes through these $\varepsilon$
original lines correspond to the points of the subspaces in the dual. Hence by {Lemma~ \ref
{defhyperpl}}, all points of these subspaces are in $f$, which means that in $[\alpha]$ there are $
\varepsilon$ different subspaces of dimension $n-2$ totally contained in $f$.
\end{proof}

Now we prove a lemma, which is interesting for its own sake as well.

\begin{lemma}
\label{tanlemma}
Let $f(X_0,...,X_n)$ be a homogeneous polynomial of degree $d<q$.
Suppose that there are $n-1$ independent concurrent lines
$\ell_1, ...,\ell_{n-1}$ through the point $P$ in $\pg(n,q)$  totally
contained in the hypersurface $f=0$. Then the hyperplane spanned by
$\ell_1,...,\ell_{n-1}$ is a tangent hyperplane of $f$.
\end{lemma}

\begin{proof}
Without loss of generality. let $P=(1,0,0,...,0)$ and $\ell_i$ be the ``axis''
$\langle P, \  (\stackrel{0}{1},\stackrel{1}{0},\stackrel{}{0},...,
\stackrel{}{0},\stackrel{i}{1}, \stackrel{}{0},\dots, \stackrel{n}{0})\rangle$, $i=1,...,n-1$.
We want to prove that the hyperplane $x_n=0$, i.e. $[0,...,0,1]$ is
tangent to $f$ at $P$.

Firstly, observe that $\partial_{X_0} f (P)=0$ as
$f$ has no term of type $X_0^d$ since $f(P)=0$.

Now we prove that $\partial_{X_i} f (P)=0$ for all $i=1,...,n-1$.
As $f$ vanishes on $\ell_i$ we have
$f(sX_i,0,...,0,X_i,0,...,0)=0$ for all substitutions to $s$ and $X_i$.
As $f(sX_i,0,...,0,X_i,0,...,0)=X_i^d f_0(s)$ for some $f_0$ with $\deg f_0\leq d<q$, we have
$f_0\equiv 0$.
In particular, $f_0$ has no term of degree $d-1$, so $f$ has no term of type
$X_0^{d-1}X_i$. Hence $\partial_{X_i} f (1,0,0,...,0)=0.$
\end{proof}

\begin{cor}
Let $f(X_0,...,X_n)$ be a homogeneous polynomial of degree $d<q$. Suppose that
in $\pg(n,q)$ the intersection of a hyperplane $H$ and the hypersurface $f=0$
contains two complete subspaces of dimension $n-2$. Then $H$ is a tangent
hyperplane of $f$.
\end{cor}
\begin{proof}
Choose a point $P$ in the intersection of the two subspaces of dimension
$n-2$, the lines $\ell_1,...,\ell_{n-2}$ through $P$ in one of the subspaces and
$\ell_{n-1}$ through $P$ in the other such that $\ell_1,...,\ell_{n-1}$ be
independent and apply Lemma \ref{tanlemma}.
\end{proof}

\begin{cor}\label{tangent}
If $(\alpha)=(0,\alpha_1,\alpha_2,...,\alpha_n)\in N\subset H_\infty$ is a non-determined direction,
then (in the dual space) the hyperplane
$[\alpha]$ is a tangent hyperplane of $f$. Note that $[\alpha]$ contains $(1,0,...,0)$.
\end{cor}

Now we generalize Theorem~\ref{th:DBG_Ball}.

\begin{theo}\label{general_n}
Let $n \geq 3$. Let $U\subset\ag(n,q)\subset\pg(n,q)$, $|U|=q^{n-1}-2$. Let
$D\subseteq H_\infty$ be the set of directions determined by $U$ and put
$N=H_\infty\setminus D$ the set of non-determined directions. Then $U$ can be
extended to a set $\bar U\supseteq U$, $|\bar U|=q^{n-1}$ determining the same
directions only, or the points of $N$ are collinear and $|N| \leq \lfloor \frac{q+3}{2} \rfloor$,
or the points of $N$ are on a (planar) conic curve.
\end{theo}
\begin{proof}
Let $n \geq 3$. The hypersurface $f=0$ is a quadric in the projective space $\pg(n,q)$.
We will investigate the hyperplanes through the point
$(1,0,\dots,0)$ that meet $f=0$ in exactly two $(n-2)$-dimensional subspaces. If the quadric $f=0$
contains $(n-2)$-dimensional subspaces, then either $n=3$ and the quadric is hyperbolic, or the
quadric must be singular, since $\lfloor (n-1)/2 \rfloor$ is an
upper bound for the dimension of the generators. If $f=0$ contains $2$ hyperplanes, then $f=0$
is the product of two linear factors, counted with multiplicity. But then, by our remark before Lemma \ref{defhyperpl},
the set $U$ can be extended.
Hence, if we suppose that the set $U$ cannot be extended, the quadric $f=0$ contains $(n-2)$-dimensional
subspaces, so it is a cone with vertex an $(n-3)$-dimensional subspace and base a (planar) conic, or it is
a cone with vertex an $(n-4)$-dimensional subspace and base a hyperbolic quadric in a 3-space. (Note that the second
one contains the case when $n=3$ and $f$ is a hyperbolic quadric itself.) Denote in
both cases the vertex by $V$.

Firstly suppose that $f=0$ has an $(n-3)$-dimensional subspace $V$ as vertex. A hyperplane $[\alpha]$
through $(1,0,\dots,0)$ containing two $(n-2)$-dimensional subspaces must contain $V$ and meets the base conic
in two points (counted with multiplicity). Hence $[\alpha]$ is one of the $(q+1)$ hyperplanes through the span of $\langle (1,0,\dots,0),\ V \rangle$,
so dually, the undetermined direction $(\alpha)$ is a point of the line, which is the intersection of
the dual (plane) of $V$ and $H_\infty$.
 When $q$ is odd, there are $\frac{q+1}{2}$, respectively $\frac{q+3}{2}$ such hyperplanes meeting the base conic,
depending on whether the vertex $V$ is projected from the point $(1,0,\dots,0)$ onto an internal point,
respectively, an external point of the base conic. When $q$ is even, there are $\frac{q}{2}$
such hyperplanes.

Secondly suppose that $f=0$ has an $(n-4)$-dimensional subspace $V$ as vertex. Now a hyperplane
$[\alpha]$ through $(1,0,\dots,0)$ contains $V$ and it meets the base quadric in two lines, i.e. a tangent plane to
this hyperbolic quadric. Hence, $[\alpha]$ is one of the $q^2+q+1$ hyperplanes through the span of $\langle (1,0,\dots,0),\ V \rangle$,
so dually, the undetermined direction $(\alpha)$ is a point of the plane, which is the intersection of
the dual (3-space) of $V$ and $H_\infty$.

Among these hyperplanes only those count, which meet the base hyperbolic quadric in two lines, i.e. those which intersect
the base 3-space in such a tangent plane of the hyperbolic quadric, which goes through the projection of $V$ from the point $(1,0,\dots,0)$.
Dually these hyperplanes form a conic, so $(\alpha)$ is a point of this conic.
\end{proof}

We consider the case when $U$ is extendible as the typical one: otherwise $N$ has a very restricted (strong) structure; although note that 
there exist examples of maximal point sets $U$, of size $q^2-2$, $q \in \{3,5,7,11\}$, not determining the points of a conic at infinity. These 
examples occur in the theory of maximal partial ovoids of generalized quadrangles, and where studied in \cite{DWT2010}, \cite{CDBS2012},
and \cite{DeBeule2012}. Non-existence of such examples for $q=p^h$, $p$ an odd prime, $h > 1$, was shown in \cite{DBG2008}.

Now we prove a general extendability theorem in the 3-space if $\varepsilon<p$.

\begin{theo}
\label{3extend}
Let $U\subset \ag(3,q)\subset \pg(2,q)$, $|U|=q^2-\varepsilon$, where $\varepsilon<p$. Let $D\subseteq H_\infty$ be the set of directions determined by $U$ and put $N=H_\infty\setminus D$ the set of non-determined directions. Then $N$ is contained in a plane curve
of degree $\varepsilon^4-2\varepsilon^3+\varepsilon$ or $U$ can be extended to a set $\bar U\supseteq U$, $|\bar U|=q^2$.
\end{theo}

\begin{proof}
We proceed as before: we define the R\'edei polynomial of $U$, then we calculate $f(X_0,X_1,X_2,X_3)$ of degree $\varepsilon$.

Finally we realize that for each triple $(\alpha,\beta,\gamma)$, if $(0,\alpha,\beta,\gamma)\in N\subset H_\infty$ is an undetermined direction then the plane
$[0,\alpha,\beta,\gamma]$, which apparently goes through the point $(1,0,0,0)$, is a tangent plane of $f$.

The tangent planes of $f$ are of the form
$$[\partial_{X_0} f(a,b,c,d), \partial_{X_1} f(a,b,c,d), \partial_{X_2} f(a,b,c,d), \partial_{X_3} f(a,b,c,d)]$$
where $(a,b,c,d)$ is a smooth point of $f$, and there are some others going through
points of $f$ where $\partial_{X_0} f=\partial_{X_1} f=\partial_{X_2} f=\partial_{X_3} f=0$. For planes of both type
containing $(1,0,0,0)$ we have $\partial_{X_0} f(a,b,c,d)=0$, so we get that the triples
$(\alpha,\beta,\gamma)$, with $(0,\alpha,\beta,\gamma)\in H_\infty$ being an undetermined direction, correspond to tangent
planes $[0,\alpha,\beta,\gamma]$ of $f$ in points $(a,b,c,d)$ which belong to the intersection of $f$ and $\partial_{X_0} f$, which
is a space curve $\C$ of degree $\varepsilon(\varepsilon-1)$. Projecting these tangent planes from $(1,0,0,0)$ (which all they contain)
onto a (fixed) plane we get that in that plane the projected images $[\alpha,\beta,\gamma]$ are tangent lines of the projected image $\hat\C$,
which is a plane curve of degree $\varepsilon(\varepsilon-1)$. So we get that the undetermined directions are contained in a
plane curve of degree $\varepsilon(\varepsilon-1)\Big(\varepsilon(\varepsilon-1)-1\Big)=
\varepsilon^4-2\varepsilon^3+\varepsilon$.
\end{proof}

To reach the total strength of this theory, we would like to use an argument stating
that it is a ``very rare'' situation that in $\pg(n,q)$ a hypersurface $f=0$ with
$d=\deg f>2$ admits a hyperplane $H$ such that the intersection of $H$ and the hypersurface
splits into $d$ linear factors, i.e. $(n-2)$-dimensional subspaces (Totally Reducible Intersection,
TRI hyperplane). We conjecture the following.

\begin{conjecture}
Let $f(X_0, X_1,...,X_n)$ be a homogeneous irreducible polynomial of degree $d>2$ and let
$F$ be the hypersurface in $\pg(n,q)$ determined by $f=0$. Then the number of TRI hyperplanes to $F$ is ``small''
or $F$ is a cone with a low dimensional base.
\end{conjecture}

By small we mean the existence of a function (upper bound) $r(d,n)$, which is independent from $q$; although
we would not be surprised if even a constant upper bound, for instance $r(d,n)=45$ would hold in general.
By a low dimensional base of a cone we mean an at most 3-dimensional base.

We remark finally that such a result would immediately imply extendability of direction sets $U$ under very general conditions.

\section{An application}

A (finite) \emph{partial geometry}, introduced by Bose \cite{Bose1963}, is an incidence structure
$\IS$ in which $\P$ and $\B$ are disjoint non-empty sets of objects called
points and lines (respectively), and for which $\I \subseteq (\P \times
\B) \cup (\B \times \P)$ is a symmetric point-line incidence relation
satisfying the following axioms:
\begin{itemize}
\item[(i)] Each point is incident with $1+t$ lines $(t \geqslant 1)$ and
two distinct points are incident with at most one line.
\item[(ii)] Each line is incident with $1+s$ points $(s \geqslant 1)$ and
two distinct lines are incident with at most one point.
\item[(iii)] There exists a fixed integer $\alpha > 0$, such that if $x$ is a point and 
$L$ is a line not incident with $x$, then there are exactly $\alpha$ pairs 
$(y_i,M_i) \in \P \times \B$ for which $x \I M_i \I y_i \I L$.
\end{itemize}
The integers $s$, $t$ and $\alpha$ are the parameters of $\S$. The \emph{dual} $\S^D$ of a
partial geometry $\IS$ is the incidence structure $(\B,\P,\I)$. It is a partial geometry
with parameters $s^D = t$, $t^D = s$, $\alpha^D = \alpha$.

If $\S$ is a partial geometry with parameters $s$, $t$ and $\alpha$, then 
$|\P| = (s+1)\frac{(st+\alpha)}{\alpha}$ and $|\B| = (t+1)\frac{(st+\alpha)}{\alpha}$. 
(see e.g. \cite{DCVM1995}). A partial geometry with parameters $s,t$, and $\alpha=1$, is a 
{\em generalized quadrangle} of order $(s,t)$, \cite{PT:2009}.

To describe a class of partial geometries of our interest, we need special pointsets
in $\pg(2,q)$. An {\em arc of degree $d$} of a projective plane $\Pi$ of order $s$ is a set $\K$ of points such that
every line of $\Pi$ meets $\K$ in at most $d$ points. If $\K$ contains $k$ points,
than it is also called a $\{k,d\}$-arc. The size of an arc of degree $d$ can not exceed $ds-s+d$.
A $\{k,d\}$-arc $\K$ for which $k=ds-s+d$, or equivalently, such that every line that meets $\K$,
meets $\K$ in exactly $d$ points, is called {\em maximal}. We call a $\{1,1\}$-arc and a $\{s^2,s\}$-arc {\em trivial}. 
The latter is necessarily the set of $s^2$ points of $\Pi$ not on a chosen line.

A typical example, in $\pg(2,q)$, is a conic, which is a $\{q+1,2\}$-arc, which is not maximal, and it is well known 
that if $q$ is even, a conic, together with its 
nucleus, is a $\{q+2,2\}$-arc, which is maximal. We mention that a $\{q+1,2\}$-arc in $\pg(2,q)$ is also called
an {\em oval}, and a $\{q+2,2\}$-arc in $\pg(2,q)$ is also called a {\em hyperoval}. When $q$ is odd, all ovals are
conics, and no $\{q+2,2\}$-arcs exist (\cite{Segre1955}). When $q$ is even, every oval has a nucleus, and 
so can be extended to a hyperoval. Much more examples of hyperovals, different from a conic and its 
nucleus, are known, see e.g. \cite{DCD:2011}.  We mention the following two
general theorems on $\{k,d\}$-arcs.

\begin{theo}[\cite{Cossu1961}]\label{th:arcs:dual}
Let $\K$ be a $\{ds-s+d,d\}$-arc in a projective plane of order $s$. Then the set of lines external to $\K$ is a 
$\{s(s-d+1)/d,s/d\}$-arc in the dual plane.
\end{theo}

As a consequence, $d \mid s$ is a necessary condition for the existence of maximal $\{k,d\}$-arcs in a 
projective plane of order $s$. The results for the Desarguesian plane $\pg(2,q)$ are much stronger.
Denniston \cite{Denniston1969} showed that this condition is sufficient for the existence of maximal $\{k,d\}$-arcs
in $\pg(2,q)$, $q$ even. Blokhuis, Ball and Mazzocca \cite{BBM1997} showed that non-trivial maximal $\{k,d\}$-arcs in $\pg(2,q)$ 
do not exist when $q$ is odd. Hence, the existence of maximal arcs in $\pg(2,q)$ can be summarized in the following
theorem.

\begin{theo}\label{th:arcs:odd}
Non-trivial maximal $\{k,d\}$-arcs in $\pg(2,q)$ exist if and only if $q$ is even.
\end{theo}

Several infinite families and constructions of maximal $\{k,d\}$-arcs of $\pg(2,q)$, $q=2^h$, and $d=2^e$, $1 \leq e \leq h$, 
are known. We refer to \cite{DCD:2011} for an overview.

Let $q$ be even and let $\K$ be a maximal $\{k,d\}$-arc of $\pg(2,q)$. We define the incidence structure $T_2^*(\K)$ as
follows. Embed $\pg(2,q)$ as a hyperplane $H_{\infty}$ in $\pg(3,q)$. The points of $\S$ are the points of $\pg(3,q)\setminus H_{\infty}$. The 
lines of $\S$ are the lines of $\pg(3,q)$ not contained in $H_{\infty}$, and meeting $H_{\infty}$ in a point of $\K$. The incidence
is the natural incidence of $\pg(3,q)$. One easily checks  
that $T_2^*(\K)$ is a partial geometry with parameters $s=q-1$, $t=k-1=(d-1)(q+1)$, and $\alpha = d-1$.

An {\em ovoid} of a partial geometry $\IS$ is a set $\O$ of points of $\S$, such that every line of $\S$ meets $\O$ in 
exactly one point. Necessarily, an ovoid contains $\frac{st}{\alpha}+1$ points. Different examples of partial geometries exist, 
and some of them have no ovoids, see e.g. \cite{DCDFG1991}. The partial geometry $T_2^*(\K)$ has always an ovoid. 
Consider any plane  $\pi \neq H_{\infty}$ meeting $H_{\infty}$ in a line skew to $\K$. The plane $\pi$ then contains 
$\frac{st}{\alpha}+1=q^2$ points of $\S$, and clearly every line of $\S$ meets $\pi$ in exactly one point.

It is a natural stability question to investigate {\em extendability} of point sets of size slightly smaller than the size of an ovoid. In this
case, the question is whether a set of points $\B$, with the property that every line meets $\B$ in at most one point, can be extended 
to an ovoid if $|\B| = q^2 - \varepsilon$, and $\varepsilon$ is {\em not too big}. Such a point set $\B$ is called a {\em partial ovoid}, 
$\varepsilon$ its {\em deficiency}, and it is called {\em maximal} if it cannot be extended. The following theorem is from \cite{PT:2009} 
and deals with this question in general for GQs, i.e. for $\alpha=1$. 

\begin{theo}\label{th:partial}
Consider a GQ of order $(s,t)$. Any partial ovoid of size $(st-\rho)$, with $0 \leq \rho < t/s$ is
contained in a uniquely defined ovoid.
\end{theo}

For some particular GQs, extendability beyond the given bound is known. For other GQs, no better bound is known, or examples of 
maximal partial ovoids reaching the upper bound, are known. For an overview, we refer to \cite{DBKM:2012}.

Applied to the GQ $T_2^*(\mathcal{H})$, $\mathcal{H}$ a hyperoval of $\pg(2,q)$, Theorem~\ref{th:partial} yields that a partial ovoid 
of $T_2^*(\mathcal{H})$ of size $q^2-2$ can always be extended. The proof of Theorem~\ref{th:partial} is of combinatorial nature, and
can be generalized to study partial ovoids of partial geometries. However, for the partial geometries $T_2^*(\K)$ with $\alpha \geq 2$, 
such an approach only yields extendability of partial ovoids with deficiency one. In the context of this paper, we can study extendability
of partial ovoids of the partial geometry $T_2^*(\K)$ as a direction problem. Indeed, if a set of points $\B$ is a (partial) ovoid, then
no two points of $\B$ determine a line of the partial geometry $T_2^*(\K)$. Hence the projective line determined by two points of $\B$
must not contain a point of $\K$, in other words, the set of points $\B$ is a set of affine points, not determining the points of $\K$ at 
infinity. 

Considering a partial ovoid $\B$ of size $q^2-2$, we can apply Theorem~\ref{general_n}. Clearly, the non-determined directions, which contain
the points of $\K$, do not satisfy the conditions when $\B$ is not extendable. Hence, we immediately have the following corollary.

\begin{cor}
Let $\B$ be a partial ovoid of size $q^2-2$ of the partial geometry $T_2^*(\K)$, then $\B$ is always extendable to an ovoid. 
\end{cor}

This result is the same as Theorem~\ref{th:partial} for the GQ $T_2^*(\mathcal{H})$, $\mathcal{H}$ a hyperoval of $\pg(2,q)$, $q > 2$.

\section*{Acknowledgements}
The authors are grateful to the anonymous referee for his valuable remarks on section 3. The first author thanks the department of 
Computer Science at E\"otv\"os Lor\'and University in Budapest, and especially P\'eter Sziklai,  Tam\'as Sz\H{o}nyi and Zsuzsa Weiner 
for their hospitality.




\begin{thebibliography}{10}

\bibitem{Ball:2011}
S.~Ball.
\newblock The polynomial method in {G}alois geometries.
\newblock In J.~De~Beule and L.~Storme, editors, {\em Current research topics
  in Galois geometry}, Mathematics Research Developments, chapter~5, pages
  105--130. Nova Sci. Publ., New York, 2012.

\bibitem{BBM1997}
S.~Ball, A.~Blokhuis, and F.~Mazzocca.
\newblock Maximal arcs in {D}esarguesian planes of odd order do not exist.
\newblock {\em Combinatorica}, 17(1):31--41, 1997.

\bibitem{Bose1963}
R.~C. Bose.
\newblock Strongly regular graphs, partial geometries and partially balanced
  designs.
\newblock {\em Pacific J. Math.}, 13:389--419, 1963.

\bibitem{CDBS2012}
K.~Coolsaet, J.~De~Beule, and A.~Siciliano.
\newblock {The known maximal partial ovoids of size $q^2-1$ of
  $\mathrm{Q}(4,q)$}.
\newblock {\em J. Combin. Des.}, 21(3): 89 -- 100, 2013.

\bibitem{Cossu1961}
A.~Cossu.
\newblock Su alcune propriet{\`a} dei {$\{k,\,n\}$}-archi di un piano
  proiettivo sopra un corpo finito.
\newblock {\em Rend. Mat. e Appl. (5)}, 20:271--277, 1961.

\bibitem{DeBeule2012}
J.~De~Beule.
\newblock On large maximal partial ovoids of the parabolic quadric
  $\mathrm{Q}(4,q)$.
\newblock {\em Des. Codes Cryptogr.}, DOI 10.1007/s10623-012-9629-y, 2012.

\bibitem{DBG2008}
J.~De~Beule and A.~G{\'a}cs.
\newblock Complete arcs on the parabolic quadric {${\rm Q}(4,q)$}.
\newblock {\em Finite Fields Appl.}, 14(1):14--21, 2008.

\bibitem{DBKM:2012}
J.~De~Beule, A.~Klein, and K.~Metsch.
\newblock Substructures of finite classical polar spaces.
\newblock In J.~De~Beule and L.~Storme, editors, {\em Current research topics
  in Galois geometry}, Mathematics Research Developments, chapter~2, pages
  35--61. Nova Sci. Publ., New York, 2012.

\bibitem{DCDFG1991}
F.~De~Clerck, A.~Del~Fra, and D.~Ghinelli.
\newblock Pointsets in partial geometries.
\newblock In {\em Advances in finite geometries and designs ({C}helwood {G}ate,
  1990)}, Oxford Sci. Publ., pages 93--110. Oxford Univ. Press, New York, 1991.

\bibitem{DCD:2011}
F.~De~Clerck and N.~Durante.
\newblock Constructions and characterizations of classical sets in $\mathrm{PG}(n,q)$.
\newblock In J.~De~Beule and L.~Storme, editors, {\em Current research topics
  in Galois geometry}, Mathematics Research Developments, chapter~1, pages
  1--33. Nova Sci. Publ., New York, 2012.

\bibitem{DCVM1995}
F.~De~Clerck and H.~Van~Maldeghem.
\newblock Some classes of rank {$2$} geometries.
\newblock In {\em Handbook of incidence geometry}, pages 433--475.
  North-Holland, Amsterdam, 1995.

\bibitem{DWT2010}
S.~De~Winter and K.~Thas.
\newblock Bounds on partial ovoids and spreads in classical generalized
  quadrangles.
\newblock {\em Innov. Incidence Geom.}, 11:19--33, 2010.

\bibitem{Denniston1969}
R.~H.~F. Denniston.
\newblock Some maximal arcs in finite projective planes.
\newblock {\em J. Combinatorial Theory}, 6:317--319, 1969.

\bibitem{PT:2009}
S.~E. Payne and J.~A. Thas.
\newblock {\em Finite generalized quadrangles}.
\newblock EMS Series of Lectures in Mathematics. European Mathematical Society
  (EMS), Z\"urich, second edition, 2009.

\bibitem{Segre1955}
B.~Segre.
\newblock Ovals in a finite projective plane.
\newblock {\em Canad. J. Math.}, 7:414--416, 1955.

\bibitem{Szonyi1996}
T.~Sz{\H{o}}nyi.
\newblock On the number of directions determined by a set of points in an
  affine {G}alois plane.
\newblock {\em J. Combin. Theory Ser. A}, 74(1):141--146, 1996.

\end{thebibliography}

\IfFileExists{DBSZT_references.tex}{%
\newwrite\outputstreamc
\immediate\openout\outputstreamc=content.tex
}{}

\newwrite\outputstreamd
\immediate\openout\outputstreamd=DBSZT_references.tex

Address of the authors:\\

\noindent Jan De Beule. Ghent University, Department of  Mathematics, Krijgslaan 281-S22, B--9000 Ghent, Belgium\\
jdebeule@cage.ugent.be, \url{http://cage.ugent.be/~jdebeule}\\
and\\
Vrije Universiteit Brussel, Department of Mathematics, Pleinlaan 2, B--1050 Brussel, Belgium\\

\noindent P\'eter Sziklai and Marcella Tak\'ats.
E\"otv\"os Lor\'and University, Institute of Mathematics, P\'azm\'any P\'eter s\'et\'any 1/C, H–1117 Budapest, Hungary\\
\{sziklai,takats\}@cs.elte.hu.

\end{document}